\documentclass[11pt,reqno]{amsart}

\usepackage{amssymb}
\usepackage{amscd}
\usepackage{amsfonts}
\usepackage{mathrsfs}
\usepackage{setspace}
\usepackage{version}

\numberwithin{equation}{subsection} \theoremstyle{plain}
\newtheorem{theorem}[subsection]{Theorem}
\newtheorem{proposition}[subsection]{Proposition}
\newtheorem{lemma}[subsection]{Lemma}

\newtheorem{problem}[subsection]{Problem}

\newcommand\R{\mathbf{R}}

\newcommand\A{\mathbf{Aff}}

\newcommand\rk{\operatorname{rk}}

\newcommand\F{\mathbf{F}}
\newcommand\Q{\mathbf{Q}}

\newcommand\eps{\varepsilon}
\newcommand\id{\operatorname{id}}

\begin{document}

\title{On conjugacy growth for solvable groups}
\subjclass[2000]{20F16 (Primary), 20F69, 20G25 (Secondary)}
\author{Emmanuel Breuillard, Yves de Cornulier}
\address{Laboratoire de Math\'ematiques\\
B\^atiment 425, Universit\'e Paris Sud 11\\
91405 Orsay\\
FRANCE}
\email{emmanuel.breuillard@math.u-psud.fr}


\address{IRMAR, Campus de Beaulieu, 35042 Rennes CEDEX, France}
\email{yves.decornulier@univ-rennes1.fr}
\date{September 12, 2010}

\begin{abstract}
We prove that a finitely generated solvable group which is not virtually nilpotent has exponential conjugacy growth.
\end{abstract}

\maketitle

This note is an addendum to the paper ``On the conjugacy growth functions of groups" by V.~Guba and M.~Sapir \cite{guba-sapir}.

\section{Introduction}

Let $G$ be a finitely generated group. Our goal is to provide a proof of the following result

\begin{theorem}\label{solvable} If $G$ is virtually solvable and not virtually nilpotent then $G$ has exponential conjugacy growth. In fact the exponential conjugacy growth rate can be bounded below by a positive number independent of the generating set.
\end{theorem}

In fact, using different techniques, Michael Hull \cite{hull} independently proved that non-virtually-nilpotent polycyclic groups have exponential conjugacy growth.

Recall that a group $G$ generated by a finite symmetric set $\Sigma$ containing $1_G$ is said to have \emph{exponential conjugacy growth}, if $$C_{G,\Sigma}:=\liminf_{n \rightarrow \infty} \frac{1}{n} \log c_{G,\Sigma}(n) >0,$$
where $c_{G,\Sigma}(n)$ is the number of conjugacy classes of $G$ intersecting the $n$-ball $\Sigma^n$ nontrivially. It is easy to see that although $C_{G,\Sigma}$ may depend on the choice of $\Sigma$, the fact that it is positive does not. If the $C_{\Sigma,G}$ admit a positive
lower bound independent of the choice of $\Sigma$ in $G$, we say that $G$ has \emph{uniform exponential conjugacy growth}. Thus Theorem \ref{solvable} can be reformulated as saying that finitely generated non-virtually-nilpotent solvable groups have uniform exponential conjugacy growth.

Recall that D.~Osin proved in \cite{osin} that finitely generated solvable groups have uniform exponential word growth unless they are virtually nilpotent. This means that the rate of exponential word growth can be bounded below by a number independent of the choice of a generating set in $G$. Theorem \ref{solvable} is thus a strengthening of Osin's result. However we will use the fact that solvable groups have uniform exponential growth in our proof of Theorem \ref{solvable}.

In \cite{breuillard} the first-named author has given another proof of Osin's theorem which relied on a ping-pong argument and made key use of a theorem of J. Groves about metabelian images of solvable groups. This result of Groves will also be instrumental in the proof of Theorem \ref{solvable}.

Let us recall Groves' theorem. Let $\mathbf{Aff}$ be the algebraic group of affine transformations of the line $\{x\mapsto ax+b\}$, that is, if $K$ is any field,
\begin{equation*}
\mathbf{Aff}(K)=\left\{ u(a,b):=\left(
\begin{array}{ll}
a & b \\
0 & 1
\end{array}
\right) ,a\in K^{\times },b\in K\right\}
\end{equation*}

\begin{theorem}(Groves \cite{groves})\label{thm-groves}
Let $G$ be a finitely generated virtually solvable group which is not virtually nilpotent. Then there exists a field $K$ and a finite index subgroup $H$ of $G$ such that $H$ admits a homomorphism into the affine group $\A(K)$ whose image is not virtually nilpotent.
\end{theorem}

For a proof of the above theorem, we refer the reader to Groves' original paper, which is based on prior work of Philip Hall, or to \cite[Theorem 1.6]{breuillard} for a self-contained proof. 

Groves's theorem will reduce the proof of Theorem \ref{solvable} to the case when $G$ is a subgroup of $\A(K)$ for some field $K$, which can be assumed to be finitely generated because $G$ is. 

Let us end this introduction by mentioning the following open problem (see also Problem \ref{open} below).

\begin{problem}Prove (or disprove) that any non-virtually-nilpotent finitely generated elementary amenable group has exponential conjugacy growth.
\end{problem}

By Chou \cite{chou}, such a group has exponential growth; however there is no analogue of Groves's theorem for those groups and accordingly our methods do not apply.
\vspace{11pt}

\section{Proof of the main theorem}
 
We now begin the proof of Theorem \ref{solvable} and we will complete it by the end of this section modulo certain auxiliary observations, which will be proven in the last two sections, the first of which being the following slight refinement of Groves's theorem.
 
Recall that a {\it global field} is a finite algebraic extension of either $\Q$ or $\F_p(t)$.

\begin{proposition}\label{global} In Groves's theorem \ref{thm-groves}, we can take $K$ to be a global field.
\end{proposition}

\begin{proof} See Section \ref{pasp}.
\end{proof}

To prove Theorem \ref{solvable}, we also need the following elementary lemma, whose proof is postponed to Section \ref{aux}.

\begin{lemma}\label{straight}
Let $G$ be a finitely generated group and $H$ either a quotient or a finite index subgroup of $G$. If $H$ has exponential conjugacy growth (resp. uniform exponential conjugacy growth), then so does $G$.
\end{lemma}

\begin{proof} See Section \ref{aux}.
\end{proof}

Theorem \ref{solvable} will follow easily from:

\begin{proposition}\label{unipotent} Let $K$ be a global field. Let $G$ be a finitely generated subgroup of $\A(K)$ which is not virtually nilpotent. Let $r$ be the rank (as an abelian group) of the image $\Gamma$ of $G$ inside $K^\times$ and $\Sigma$ a finite symmetric generating subset of $G$ containing $\id$. Then, the intersection of $\Sigma^n$ with any unipotent conjugacy class of $G$ has cardinality $O_{\Sigma}(n^r)$ (uniformly over the conjugacy class).
\end{proposition}

Since $G$ is finitely generated, the matrix entries of its elements lie in a finitely generated subring $A$ of $K$. Such a ring $A$ can be embedded (diagonally) as a discrete subring into a finite product $\prod_{i=1}^d\mathbf{K}_i$ of local fields (Archimedean or not), so that the product formula is satisfied, namely $\prod\alpha_i(a)=1$ for all $a\in A^\times$, where $\alpha_i(x)$ is the absolute value of $x \in \mathbf{K}_i$, and $A^\times$ the group of invertible elements in $A$.

Setting $\alpha(a)= \max_i \alpha_i(a)$ we see from the product formula that $\alpha(a^{-1}) \leq \alpha(a)^{d-1}$ for all $a \in A^{\times}$.

The map $\prod_{i=1}^d\mathbf{K}_i^{\times} \rightarrow \R^d, (x_i)_i \mapsto (\log(\alpha(x_i))_i$ is clearly a proper map, and hence so is the group homomorphism $\Gamma \rightarrow \R^d, \gamma \mapsto (\log(\alpha_i(\gamma)))_i$, where $\Gamma$ is endowed with the discrete topology. Its image is a discrete subgroup of $\R^d$ of rank at most $r=\rk(\Gamma)$ and consequently, the number of elements $\gamma \in \Gamma$ with $\alpha(\gamma) \leq \exp(n)$ is $O(n^r)$.

\begin{proof}[{\bf Proof of Proposition \ref{unipotent}}]

The image $\Gamma$ of $G$ in $K^{\times}$ lies in $A^{\times}$.
Since the conjugation action on the set of unipotent elements factors through $A^{\times}$, we see that two elements $u(1,b)$ and $u(1,b')$ are conjugate in $G$ if and only if there exists $\gamma\in\Gamma$ such that $b'=\gamma b$.

It is easy to check that there exists some $C=C_{\Sigma}>0$ such that if $u(a,b)\in\Sigma^n$, then $\alpha(b)\le\exp(Cn)$. On the other hand $\alpha(\gamma) \leq \alpha(b^{-1})\alpha(b') \leq \alpha(b)^{d-1}\alpha(b') \leq \exp(dCn)$. According to the remark preceding the proof of the proposition, there are only $O(n^r)$ such elements in $\Gamma$. We are done.
\end{proof}

\begin{proof}[{\bf Proof of Theorem \ref{solvable}}]
Applying Groves's theorem combined with Proposition \ref{global} and Lemma \ref{straight}, we may assume that $G$ is a finitely generated non-virtually-nilpotent subgroup of $\A(K)$, where $K$ is a global field. Let $\Sigma$ be a finite symmetric generating subset of $G$ containing $\id$.

Applying Osin's theorem on uniform exponential growth, or rather the more precise Theorem 1.3 of \cite{breuillard}, we know that $\Sigma^n$ contains at least $(1+\eps)^n$ elements, where $\eps$ is positive and independent of the choice of $\Sigma$. On the other hand, the image $\Gamma$ of $G$ in $K^{\times}$ is a finitely generated abelian group. Let $r$ be its rank. The size of the image of $\Sigma^{n/2}$ is therefore at most $R_{\Sigma}(n)=O_{\Sigma}(n^r)$. By the pigeonhole principle, there must be one value in $K^{\times}$ with at least $(1+\eps)^{n/2}/R_{\Sigma}(n)$ preimages in $\Sigma^{n/2}$. Given $g,h$ two such preimages $gh^{-1}$ is a unipotent element in $\Sigma^n$. This makes at least $(1+\eps)^{n/2}/R_{\Sigma}(n)$ unipotent elements in $\Sigma^n$.

Applying Proposition \ref{unipotent}, there are at most $O_{\Sigma}(n^r)$ elements of $\Sigma^n$ in every given unipotent conjugacy class of $G$. Hence for all $n$ large enough (i.e. $\geq n_0(\Sigma)$), one can find at least $(1+\frac{\eps}{2})^{n/2}$ different conjugacy classes in $\Sigma^n$. This ends the proof.
\end{proof}

\section{General facts regarding conjugacy growth}\label{aux}

We prove here Lemma \ref{straight}. In fact, it immediately follows from the following more general result. If $r\ge 0$ is real, define $c_{G,\Sigma}(r)=c_{G,\Sigma}(\lfloor r\rfloor)$, which is the maximal number of non-$G$-conjugate elements in the $r$-ball in $G$.

\begin{lemma}\label{straighte}

1) Let $G$ be a finitely generated group and $H$ a quotient of $G$. Let $\Sigma$ be a generating subset (finite symmetric and containg $\id$), and $\Sigma'$ the image of $\Sigma$ into $H$.
Then for all $n$ we have $c_{G,\Sigma}(n)\ge c_{H,\Sigma'}(n)$.

2) Let $G$ be a finitely generated group and $H$ a finite index subgroup of index $k$. Let $\Sigma$ be a generating subset (finite symmetric and containing $\id$). Then there exists a generating subset $\Sigma'$ of $H$ (finite symmetric and containing $\id$), such that, for all $n$ we have
$$c_{G,\Sigma}(n)\ge \frac{1}{k}c_{H,\Sigma'}\left(\frac{n}{2k+1}\right).$$
\end{lemma}
\begin{proof}
The assertion for $H$ a quotient of $G$ is left to the reader.

If $H$ has index $k$ in $G$, it is well-known (see e.g. \cite[Lemma C.1.]{bgt} for a short proof or deduce it from \cite[Lemma~2.2 and Theorem~2.7]{MKS}) that
$\Sigma^{2k-1}$ contains a symmetric generating subset $\Sigma'$ of $H$. So $\Sigma^n$ contains the $(n/(2k-1))$-ball of $(H,\Sigma')$. In particular, it contains a subset $Y$ consisting of $c_{H,\Sigma'}(n/(2k-1))$ pairwise non-$H$-conjugate elements. Now it is clear that the maximal number of pairwise $G$-conjugate elements in $Y$ is $k$, so $Y$ contains a subset of at least $c_{H,\Sigma'}(n/(2k-1))/k$ pairwise non-$G$-conjugate elements.
\end{proof}

\begin{problem}\label{open}Prove (or disprove) that the converse of Lemma \ref{straight} does not hold, i.e.\ that exponential conjugacy growth is {\it not} inherited by finite index subgroups.
\end{problem}

\section{The specialization argument}\label{pasp} 

In this section, we prove Proposition \ref{global}. The proof relies on the following specialization argument.

\begin{lemma}\label{spero}
Let $A$ be a finitely generated commutative ring with no zero divisors. Suppose that $x,y\in A$, $x$ is not a root of unity and $y\neq 0$. Then there exists a global field $L$ and a ring homomorphism $\phi:A\to L$ such that $\phi(x)$ is not a root of unity and $\phi(y)\neq 0$.
\end{lemma}

First we explain how to deduce Proposition \ref{global} from the last lemma, and then proceed to the proof of that lemma. We shall need the following observation:

\begin{lemma} Let $\Gamma$ be a subgroup of $\mathbf{Aff}(K)$. Then $\Gamma$ is virtually nilpotent if and only if the image of $\Gamma$ in the multiplicative group $K^\times$ is finite or $\Gamma\cap K=\{0\}$.\label{vn}
\end{lemma}
 
\begin{proof}
If one of the two conditions holds, clearly $\Gamma$ is virtually abelian.

Conversely, suppose that $\Gamma$ has a finite index nilpotent subgroup $\Lambda$ and that $\Gamma\cap K$ is nontrivial. Suppose by contradiction that
the image of $\Gamma$ in $K^\times$ is infinite. Then the conjugates of any nontrivial element in $\Gamma \cap K$ give at least as many elements as there are in the image of $\Gamma$ in $K^{\times}$. Hence $\Gamma\cap K$ and hence $\Lambda \cap K$ is infinite as well.
On the other hand, since the centralizer in $\mathbf{Aff}(K)$ of any nontrivial element is abelian, and any nontrivial nilpotent group has a nontrivial center, we deduce that any nilpotent subgroup of $\mathbf{Aff}(K)$ is abelian; in particular, $\Lambda$ is abelian; since moreover the centralizer in $\mathbf{Aff}(K)$ of any nontrivial element of $K$ is contained in $K$, we deduce that $\Lambda\subset K$. So the image of $\Gamma$ in $K^\times$ is finite, a contradiction.
\end{proof}

\begin{proof}[Proof of Proposition \ref{global}]
Let $A$ be the ring generated by matrix entries of $G$. Let $x\in A^\times$ (the multiplicative group of invertible elements in $A$) be an element of infinite order in the projection of $G$ on $K^\times$, and let $u\in A$ be a nontrivial element in $G\cap K$. By Lemma \ref{spero}, there exists a homomorphism from $A$ to a global field $L$, so that the image of $x$ is not a root of unity and the image of $u$ is nonzero. By Lemma \ref{vn}, the image of $G$ in $\A(L)$ is not virtually nilpotent.
\end{proof}

We now begin the proof of Lemma \ref{spero} with the following observation.

\begin{lemma}\label{spede}Let $K$ be an infinite field, $\hat{K}$ an algebraic closure, and $A$ a finitely generated commutative $K$-algebra. Let $A$ be integral over a free subring $B=K[t_1,\dots,t_d]$, and suppose that $u\in A$ has degree $\le k$ over $B$. Then there exists a $K$-algebra homomorphism $\phi:A\to\hat{K}$ mapping $B$ into $K$ and mapping $u$ to a nonzero element. In particular, if $a\in A$ has integral degree $k$ over $B$, then $\phi(a)$ has degree $\le k$ over $K$.
\end{lemma}

\begin{proof}
We have $u^k+p_{k-1}u^{k-1}+\dots+p_0=0$, where $p_i\in K[t_1,\dots,t_d]$. Since $K$ is infinite, there exists $(u_1,\dots,u_d)\in K^d$ such that $p_0(u_1,\dots,u_d)\neq 0$. By the lifting theorem for integral extensions (see \cite[VII.3.1]{lang2}), the $K$-algebra homomorphism $K[t_1,\dots,t_d]\to K$ mapping $t_i$ to $u_i$ extends to a homomorphism $\phi:x\mapsto\bar{x}$ from $A$ to the algebraic closure of $K$.
We have
$$\bar{u}(\bar{u}^{k-1}+\phi(p_{k-1})\bar{u}^{k-2}+\dots+\phi(p_1))=-\phi(p_0),$$
so $\bar{u}\neq 0$. The last assertion is clear since $\phi(B)\subset K$.
\end{proof}

\begin{lemma}\label{fira}
Let $K$ be a finitely generated field and $\hat{K}$ an algebraic closure. Then for all $k$, the set of roots of unity in $\hat{K}$ having degree $\le k$ over $K$, is finite.
\end{lemma}
\begin{proof}This is very classical. Clearly, if $K_0\subset K$ is a finite extension, it is enough to prove the result for $K_0$. Taking $K_0$ to be a purely transcendental extension of $F=\mathbf{F}_p$ or $\mathbf{Q}$ \cite[VIII.1]{lang2}, we can therefore suppose that $K=F(t_1,\dots,t_n)$. Now if an element has degree $\le k$ over $F(t_1,\dots,t_n)$ and is algebraic over $F$, then it has degree $\le k$ over $F$ (see the proof of \cite[V.\S 11.7, Cor.~1]{Bou}). This reduces to the case when $K=F$; if $F$ is finite this is trivial since there are finitely many elements of each degree in $\hat{F}$; if $F=\mathbf{Q}$ this follows from the fact that the degree of a primitive $n$th root of unity is given by Euler's totient function $\varphi(n)$ (see \cite[V.\S 11.4]{Bou}).
\end{proof}


\begin{proof}[{\bf Proof of Lemma \ref{spero}}]
Define $(R,K)=(\mathbf{Z},\mathbf{Q})$ if $A$ has zero characteristic, and $(R,K)=(\mathbf{F}_p[x],\mathbf{F}_p(x))$ if $A$ has characteristic $p>0$ (in which case $x$ is by hypothesis transcendental over $\mathbf{F}_p$). Then $A$ embeds into $A'=A\otimes_R K$, which is the localization of $A$ with respect to $R-\{0\}$. Let $\hat{K}$ be an algebraic closure of $K$.

By the Noether normalization theorem (\cite[VIII.2.1]{lang2}), $A'$ is a finite integral extension of some $B:=K[t_1,\dots,t_n]$, where $t_1,...,t_n$ are algebraically independent indeterminates.
Let $k$ be the minimal degree of a monic polynomial over $B$ vanishing at $x$. By Lemma \ref{fira}, there are only finitely many roots of unity in $\hat{K}$ of degree at most $k$ over $K$. Let $\ell$ be the least common multiple of their orders.

Set $u=(x^\ell-1)y$. By Lemma \ref{spede}, there exists a $K$-algebra homomorphism $\phi:A'\to\hat{K}$ such that $\phi(u)\neq 0$ and $\phi(B)\subset K$. Then $\phi(y) \neq 0$ and $\phi(x)$, which has degree at most $k$ over $K$, is not a root of unity, for otherwise its order would divide $\ell$, whilst $\phi(x)^{\ell} \neq 1$.
\end{proof}

\emph{Acknowlegdements.} It is a pleasure to thank both referees for their careful reading of the manuscript and for their thoughtful suggestions.

\setcounter{tocdepth}{1}

\end{document}